\documentclass[12pt,reqno]{amsart}
\usepackage[utf8]{inputenc}
\usepackage{amsmath,amsthm}

\usepackage{amssymb}

\usepackage{enumerate}

\makeatletter
\@namedef{subjclassname@2020}{%
  \textup{2020} Mathematics Subject Classification}
\makeatother

\newtheorem{theorem}{Theorem}
\newtheorem{corollary}[theorem]{Corollary}
\newtheorem{lemma}[theorem]{Lemma}
\newtheorem{proposition}[theorem]{Proposition}
\newtheorem{maintheorem}{Theorem}
\newtheorem*{mainquestion}{Question}

\theoremstyle{remark}
\newtheorem*{remark}{Remark}

\numberwithin{equation}{section}

\newcommand{\bC}{\mathbb{C}}
\newcommand{\bP}{\mathbb{P}}
\newcommand{\bQ}{\mathbb{Q}}
\newcommand{\bZ}{\mathbb{Z}}
\newcommand{\cG}{\mathcal{G}}
\newcommand{\ord}{\textup{ord}}
\newcommand{\PGL}{\textup{PGL}}
\newcommand{\Aut}{\textup{Aut}}

\begin{document}


\baselineskip=17pt


\title[Misiurewicz poly. for rational maps with nontrivial aut.]{Misiurewicz polynomials for rational maps with nontrivial automorphisms}

\author{Minsik Han}
\address{Department of Mathematics, Box 1917, Brown University, Providence, RI 02912, USA}
\email{minsik\_han@brown.edu}

\date{}

\begin{abstract}
In this paper, we consider a one-parameter family of degree $d\ge 2$ rational maps with an automorphism group containing the cyclic group of order $d$. We construct a polynomial whose roots correspond to parameter values for which the corresponding map is post-critically finite with a certain dynamical portrait. Then we prove that the polynomial is irreducible in certain cases.
\end{abstract}



\maketitle

\section{Introduction and main results}
\label{section:intro}

A rational map is called post-critically finite if the union of the forward orbits of its critical points is finite. The study of post-critically finite rational maps is one of main topics in arithmetic dynamics. Specifically, if we have a family of rational maps parametrized by a variable (or variables), it is natural to ask which values of the variable(s) make the map be post-critically finite with a certain dynamical portrait.

For example, a polynomial map $\varphi_c(z)=z^2+c$ has a unique critical point $0$, so this map is post-critically finite if and only if $0$ is preperiodic. In particular, $0$ is periodic if and only if $c$ is a root of a polynomial $\varphi_c^n(0)=0$ for some $n$. Gleason proved that this polynomial, called \emph{Gleason polynomial}, has simple roots for all $n$, and Epstein \cite[Appendix]{epstein} and Hutz and Towsley \cite{hutz} extended this result to arbitrary degrees and strictly preperiodic cases. 

Once the polynomial is defined, it is natural to ask if it is irreducible. There were a lot of studies about irreducibility of the Gleason polynomial associated to the family of unicritical polynomial maps $z \mapsto z^d+c$, but it is still open in general. For example, Goksel \cite{goksel} proved that the Gleason polynomial is irreducible in certain cases, which is generalized by Buff, Epstein, and Koch \cite{buff}. On the other hand, Buff \cite{buff2} gave an example of reducible Gleason polynomials. \cite{buff2} also dealt with polynomials which arise in strictly preperiodic cases, called \emph{Misiurewicz polynomials}, as well as Goksel \cite{goksel2}.

However, there is not much research about rational maps, especially of higher degree, since in most cases there are lots of critical points whose behaviors are different. However, if we consider a family of rational maps which have certain nontrivial automorphisms, those automorphisms may restrict the dynamical behavior of critical points. In this manner, we are able to study the values of parameters which make such maps post-critically finite with a certain dynamical portrait.

Specifically, we consider a one-parameter family of degree $d$ rational maps: \[\left\{\phi_a (z) = \frac{az}{z^d+(d-1)}:a\ne 0\right\}\] whose automorphism group contains $C_d$, the cyclic group of order $d$. In Section \ref{section:construction}, we construct Misiurewicz polynomial $G_m \in \bZ[a]$ whose roots are the $a$-values such that each finite critical point $\gamma$ of $\phi_a$ is preperiodic, with $\phi_a^m(\gamma)$ fixed. Then we prove the following.

\begin{maintheorem}
\label{mainthm}
Suppose that $d\ge 3$ is prime. Then the $m$\textsuperscript{th} Misiurewicz polynomial $G_m$ is irreducible over $\bQ$ for all $m\le 3$.
\end{maintheorem}

This theorem could be generalized to the following question.

\begin{mainquestion}
\label{mainq}
For any positive integer $d\ge 2$ and any $m\ge 1$, is the $m$\textsuperscript{th} Misiurewicz polynomial $G_m$ irreducible over $\bQ$?
\end{mainquestion}

Calculation shows that this is true for all $d\le 12$ and $m\le 4$.

\section{Construction of Misiurewicz polynomial}
\label{section:construction}

Let $\phi:\bP^1 \to \bP^1$ be a rational map. For $f\in \PGL_2$, the $\PGL_2$-conjugation of $\phi$ is defined as \[\phi^f := f^{-1} \circ \phi \circ f.\] $f$ is called an \emph{automorphism} of $\phi$ if $\phi^f = f$, or equivalently $\phi \circ f = f\circ \phi$. The set of automorphisms of $\phi$ is a subgroup of $\PGL_2$, which is called the \emph{automorphism group} of $\phi$ and denoted by $\Aut(\phi)$. We are often interested in rational maps with nontrivial automorphisms, where $\Aut(\phi)$ has nontrivial elements other than the trival automorphism $f=\textup{id}$.

Now let $\{\phi_a:a\ne 0\}$ be a family of rational maps of degree $d\ge 2$, where \[\phi_a([x,y])=[axy^{d-1},x^d+(d-1)y^d].\] For a primitive $d$\textsuperscript{th} root of unity $\zeta_d$, we can check that \[[x,y]\mapsto [\zeta_d^i x,y],\ \ i=0,1,\cdots,d-1\] are automorphisms of $\phi$  since $\phi_a([\zeta_d^i x,y]) = \zeta_d^i \phi_a([x,y])$. Therefore, $\Aut(\phi_a)$ contains a cyclic subgroup of order $d$, generated by $[x,y] \mapsto [\zeta_d x,y]$.

\begin{remark}
In fact, Miasnikov, Stout, and Williams \cite{miasnikov} proved that any rational map of degree $d\ge 2$ with an automorphism group containing a cyclic subgroup of order $d$ is $\PGL_2$-conjugate to $\phi_a$ for some $a$.
\end{remark}

Now we investigate the forward orbits of critical points of $\phi_a$. The map $\phi_a$ has $(d+1)$ fixed points \begin{equation} \label{fixed} [0,1],\ \ [\sqrt[d]{a-d+1},1],\ \ [\zeta_d \sqrt[d]{a-d+1},1],\ \ \cdots,\ \ [\zeta_d^{d-1} \sqrt[d]{a-d+1},1]\end{equation} and $(d+1)$ critical points \[z=[1,0],\ \ [1,1],\ \ [\zeta_d,1],\ \ \cdots,\ \ [\zeta_d^{d-1},1].\] The trivial critical point $[1,0]$ is always preperiodic, since $\phi_a([1,0])=[0,1]$ is fixed. Also, since $[x,y]\mapsto [\zeta_d x,y]$ generates $C_d$ in the automorphism group, all other critical points have the same orbit structure.

We now define Misiurewicz polynomials associated to this family of maps. First, let $P_0 (x,y) = x$ and $Q_0(x,y) = y$. Then recursively define \begin{equation} \label{eq17} P_{n+1}=aP_n Q_n^{d-1},\ \ Q_{n+1}=P_n^d+(d-1)Q_n^d .\end{equation} In this way, the $n$\textsuperscript{th} iterate of $\phi_a$ is given by $\phi_a^n([x,y])=[P_n(x,y),Q_n(x,y)]$. The $n$\textsuperscript{th} \emph{dynatomic polynomial} is defined by \[\Phi_n(x,y)=\prod_{k\mid n} (Q_k x- P_k y)^{\mu(n/k)}\] where $\mu$ is the Mobius $\mu$-function.

Next, we define the \emph{generalized dynatomic polynomial} \[\Phi_{m,n}(x,y)=\begin{cases}\Phi_n(P_m,Q_m)/\Phi_n(P_{m-1},Q_{m-1}) &  \text{if } m >0, \\ \Phi_n(x,y) &  \text{if } m=0.\end{cases}\] Finally, we define the \emph{pre-Misiurewicz polynomial} as \[\widetilde{G}_{m,n}=\Phi_{m,n}(1,1).\] For example, \begin{align*} \Phi_1(x,y) &=(x^d+(d-1)y^d)x-(axy^{d-1})y\\&=x^{d+1}-(a-d+1)xy^d\\&=x(x^d-(a-d+1)y^d)\end{align*} so \[\widetilde{G}_{0,1}=\Phi_1(1,1)=1-(a-d+1)=-(a-d).\] From this, we see that the parameter value $a=d$ gives a rational function \[\phi_d([x,y]) = [dxy^{d-1},x^d+(d-1)y^d]\] such that every finite critical point is fixed.

\begin{remark}
A root $P=[x,y]\in \bP^1$ of the $n$\textsuperscript{th} dynatomic polynomial $\Phi_n$ is called to have \emph{formal period} $n$. A point $P$ has formal period $n$ if it has primitive period $n$, but the converse is not true in general. Also, the roots of $\Phi_{m,n}$ are the points whose $m$\textsuperscript{th} iterate has formal period $n$. See \cite[Section 4.1]{silverman} for the basic theory of dynatomic polynomials.
\end{remark}

We are currently interested in the nontrivial pre-fixed cases; where $m>0$ and $n=1$. In these cases, we can define \begin{align*} \widetilde{G}_m :=\widetilde{G}_{m,1}&=\Phi_{m,1}(1,1)\\&=\frac{Q_{m+1}(1,1) P_m(1,1)-P_{m+1}(1,1) Q_m(1,1)}{Q_m(1,1)P_{m-1}(1,1)-P_m (1,1)Q_{m-1}(1,1)}\\&=\frac{q_{m+1} p_m - p_{m+1}q_m}{q_mp_{m-1}-p_mq_{m-1}}\end{align*} where \[p_n = P_n(1,1),\ \ q_n = Q_n(1,1) \in \bZ[a].\] Note that $p_0=q_0=1$. We now prove the following proposition:

\begin{proposition}
\label{pgnga}
Suppose that $m\ge 1$. For $d\ge 2$, we have
\[\widetilde{G}_m=a(a-d)q_{m-1}^{d-1} G_m\] where \[G_m = \frac{1}{a-d} \left[ a^d q_{m-1}^{d(d-1)} - (a-d+1) \frac{q_m^d - (aq_{m-1}^d)^d}{q_m - aq_{m-1}^d}\right]\in \bZ[a]. \]
\end{proposition}

We call $G_m$ the $m$\textsuperscript{th} \emph{Misiurewicz polynomial}. Excluding two common roots $a=0$ and $a=d$, other roots of $\widetilde{G}_m$ are roots of either $q_{m-1}$ or $G_m$. They are precisely the $a$-values such that each finite critical point $\gamma$ of $\phi_a$ is preperiodic, with $\phi_a^m(\gamma)$ fixed. In fact, $\phi_a^m(\gamma)=0$ if $a$ is a root of $q_{m-1}$, while $\phi_a^m(\gamma)$ is one of the nonzero fixed points in \eqref{fixed} if $a$ is a root of $G_m$. While it might be interesting to deal with the factor $q_{m-1}$, we focus on $G_m$ in this paper.

\begin{remark}
In the survey paper \cite{benedetto}, this polynomial is also called Gleason polynomial. However, we follow \cite{buff2} and \cite{goksel2} where the term ``Gleason polynomial'' was used only for periodic cases.
\end{remark}

Here we state our main theorem again.

\setcounter{maintheorem}{0}

\begin{maintheorem}
Suppose that $d\ge 3$ is prime. Then the $m$\textsuperscript{th} Misiurewicz polynomial $G_m$ is irreducible over $\bQ$ for all $m\le 3$.
\end{maintheorem}

\begin{proof}[Proof of Proposition \ref{pgnga}]
Using the definitions of $p_n$ and $g_n$ and the recursion in \eqref{eq17}, we have \begin{equation} \label{eq18} p_{n+1}=ap_nq_n^{d-1},\ \ q_{n+1}=p_n^d+(d-1)q_n^d.\end{equation} Therefore, \begin{align*}\widetilde{G}_m&=\frac{q_{m+1}p_m-p_{m+1}q_m}{q_mp_{m-1}-p_mq_{m-1}}\\&=\frac{(p_m^d+(d-1)q_m^d)p_m-ap_m q_m^d}{(p_{m-1}^d+(d-1)q_{m-1}^d)p_{m-1}-ap_{m-1}q_{m-1}^d}\\&=\frac{p_m(p_m^d-(a-d+1)q_m^d)}{p_{m-1}(p_{m-1}^d-(a-d+1)q_{m-1}^d)} \\&=\frac{aq_{m-1}^{d-1}(p_m^d-(a-d+1)q_m^d)}{p_{m-1}^d-(a-d+1)q_{m-1}^d} \\&= \frac{aq_{m-1}^{d-1}((ap_{m-1}q_{m-1}^{d-1})^d-(a-d+1)(p_{m-1}^d+(d-1)q_{m-1}^d)^d)}{p_{m-1}^d-(a-d+1)q_{m-1}^d} \\&=\frac{aq_{m-1}^{d-1}(a^d p_{m-1}^d (q_{m-1}^d)^{d-1}-(a-d+1)(p_{m-1}^d+(d-1)q_{m-1}^d)^d)}{p_{m-1}^d-(a-d+1)q_{m-1}^d}. \end{align*}

For convenience, here we use substitutions $X=p_{m-1}^d$ and $Y=q_{m-1}^d$. Then \begin{align*} \frac{\widetilde{G}_m}{aq_{m-1}^{d-1}} =& \frac{a^d X Y^{d-1}-(a-d+1)(X+(d-1)Y)^d}{X-(a-d+1)Y} \\  =& \frac{1}{X-(a-d+1)Y} \Bigl[(a^d X Y^{d-1}-a^d(a-d+1)Y^d) \Bigr. \\ & \Bigl. - (a-d+1)((X+(d-1)Y)^d-a^d Y^d) \Bigr]\\ =& a^d Y^{d-1} - (a-d+1)\frac{(X+(d-1)Y)^d-(aY)^d}{(X+(d-1)Y)-aY}. \end{align*} Substituting back and using $X+(d-1)Y=p_{m-1}^d + (d-1)q_{m-1}^d = q_m$, we get \[\widetilde{G}_m=aq_{m-1}^{d-1} \left[ a^d q_{m-1}^{d(d-1)} - (a-d+1) \frac{q_m^d - (aq_{m-1}^d)^d}{q_m- aq_{m-1}^d}\right].\]

Therefore, it suffices to prove that
\begin{equation}
    \label{eq1}
    (a-d) \left| \frac{\widetilde{G}_m}{aq_{m-1}^{d-1}}\right. ,
\end{equation}
where \[\frac{\widetilde{G}_m}{aq_{m-1}^{d-1}} = a^d q_{m-1}^{d(d-1)} - (a-d+1) \frac{q_m^d - (aq_{m-1}^d)^d}{q_m- aq_{m-1}^d}.\] However, if $a=d$ we can prove that \[p_n=q_n=dq_{n-1}^d\] for all $n$ by induction from \eqref{eq18}, so \[\left.\frac{\widetilde{G}_m}{aq_{m-1}^{d-1}}\right|_{a=d} = d^d q_{m-1}^{d(d-1)} - \sum_{i=0}^{d-1} q_m^i (dq_{m-1})^{d-1-i} = d^d q_{m-1}^{d(d-1)} - d (dq_{m-1}^d)^{d-1}=0,\] which implies \eqref{eq1}.
\end{proof}

For example, here are the first few Misiurewicz polynomials for $d=3$. \begin{align*} G_1=&-a-6 \\ G_2=& -a^6 - a^5 - 30a^4 - 144a^3 - 216a^2 - 648a - 1944 \\ G_3 =& 8a^{17} - 705a^{16} - 3573a^{15} - 2295a^{14} - 124983a^{13} - 1142586a^{12} \\& - 1070172a^{11}  - 9587808a^{10}  - 41518008a^9 - 48341448a^8  \\& - 259815600a^7 - 1009029312a^6 - 1088391168a^5- 3265173504a^4 \\ & - 9795520512a^3 - 7346640384a^2 - 22039921152a - 66119763456\end{align*}

\section{Newton Polygons}
\label{section:newton}

In the proof of Theorem \ref{mainthm}, we will use the theory of Newton polygons. For a prime $p$ and a polynomial $f(z)=\sum_{i=0}^n a_i z^i \in \bC_p [z]$ where $a_n \ne 0$, we fix the following notations which will be used throughout this paper.
\begin{itemize}
    \item The $i$\textsuperscript{th} coefficient of $f$ is denoted as $c_i(f):=a_i$.
    \item Its $p$-adic valuation in $\bC_p$ is denoted as \[v_{i,p}(f):=\ord_p(c_i(f))=\ord_p(a_i).\] When $p$ is obvious, we often write just $v_i(f):=v_{i,p}(f)$.
\end{itemize}
Then the \emph{Newton polygon} of $f$ is defined as \[N(f) := \text{Lower convex hull of }\{(i,v_i(f)):i=0,1,2,\cdots,n\}\] defined on $0\le x \le n$. By definition, the Newton polygon is composed of line segments of increasing slopes. (If $a_0=0$, the first line segment may have slope $-\infty$.) The points where the slope changes, along with two endpoints $(0,v_0(f))$ and $(n,v_n(f))$, are called \emph{vertices} of the Newton polygon. Note that any given set of vertices determines at most one Newton polygon. The slopes of the Newton polygon gives a nice explanation about roots of the associated polynomial.

\begin{proposition}
\label{newtonroot}
Consider the Newton polygon of $f$ as above. If a line segment in the Newton polygon has slope $-m$ and horizontal length $\ell$, then exactly $\ell$ roots, counted with multiplicity, of $f$ in $\bC_p$ have $p$-adic valuation $m$.
\begin{proof}
See \cite[IV.4, Lemma 4]{koblitz}. The proof assumes that the constant term is $1$, but the same method can be applied to the case where the constant term is nonzero. Furthermore, using the convention $\ord_p(0)=\infty$ we can extend the proof to the general case. 
\end{proof}
\end{proposition}

In other words, if we make two multisets, one with roots of $f$ (counted with multiplicity) and another with slopes of line segments in $N(f)$ (counted with horizontal length), then these two multisets have a natural one-to-one correspondence. This idea is the key to understanding an important property of Newton polygons.

\begin{proposition}
\label{newtonprod}
Let $f_i(z)=\sum_j a_{i,j} z^j \in \bC_p[z]$ be a finite number of polynomials with nonzero constant terms. Then the Newton polygon of $\prod f_i$ is the `rearranged concatenation' of Newton polygons of $f_i$'s, constructed by the following method: we gather all line segments of all Newton polygons, arrange them by slope in increasing order, and attach them from the starting point $(0 ,\sum v_0(f_i))$.
\begin{proof}
First, it is obvious that \[v_0 \left(\prod f_i\right) = \ord_p \left(\prod a_{i,0} \right) = \sum \ord_p (a_{i,0}) = \sum v_0(f_i).\] On the other hand, the multiset of roots of $\prod f_i$ is the sum of multisets of roots of $f_i$, so the multiset of slopes of $N(\prod f_i)$ is also the sum of multisets of slopes of $N(f_i)$. This multiset of slopes of $N(\prod f_i)$ and the starting point $(0,\sum v_0(f_i))$ determines a unique Newton polygon, which is equal to that described in the statement.
\end{proof}
\end{proposition}

There are some applications of this theory to polynomials with rational coefficients. Here we define the \emph{$p$-Newton polygon} of $f(z) \in \bQ[z]$, denoted by $N_p (f)$, as the Newton polygon of $f(z)$ with the embedding $\bQ \hookrightarrow \bC_p$. It is actually equivalent to the lower convex hull of the set of points as above, using the $p$-adic valuation in $\bQ$.

\begin{corollary}
\label{newtonfactor}
Suppose that $f(z)\in \bQ[z]$ has nonzero constant term. If the $p$-Newton polygon of $f(z)$ contains exactly $k$ lattice points, i.e., $N_p(f) \cap \bZ^2$ contains exactly $k$ points, then there are at most $k-1$ factors of $f$ over $\bQ$.
\begin{proof}
Suppose that $f(z)$ can be represented as a product of $k$ factors in $\bQ[z]$. We can consider this factorization in $\bC_p[z]$. Then by Proposition \ref{newtonprod} the $p$-Newton polygon of $f$ should be the rearranged concatenation of $p$-Newton polygons of those factors. However, since two endpoints of the $p$-Newton polygon of any polynomial in $\bQ[z]$ are always lattice points, so the rearranged concatenation has at least $k+1$ lattice points, contradiction.
\end{proof}
\end{corollary}

\begin{corollary} [Eisenstein's criterion]
\label{eisenstein}
If the $p$-Newton polygon of $f(z) \in \bQ[z]$ is composed of only one line segment, with no lattice point except for the vertices, then $f(z)$ is irreducible over $\bQ$.
\end{corollary}

In this paper, we often execute delicate calculation on valuations of coefficients. For convenience, here we introduce a useful lemma that we will use afterward.

\begin{lemma}
\label{orderlemma}
Let $f(z)=a_n z^n + \cdots + a_0$ be a polynomial, and consider the $k$\textsuperscript{th} power $f(z)^k$. Let $(\alpha_1,\cdots,\alpha_k)$ be a $k$-tuple of integers satisfying $0\le \alpha_1 \le \cdots \le \alpha_k \le n$, and let $N(\alpha_1,\cdots,\alpha_k)$ be the number of permutations of the subscripts maintaining the ordering. That is, $N(\alpha_1,\cdots,\alpha_k)$ is the number of permutations $(r_1,\cdots,r_k)$ of $(1,\cdots,k)$ such that $\alpha_{r_1} \le \cdots\le \alpha_{r_k}$. Then
\begin{equation} \label{eq13} v_i(f(z)^k)\ge \min_{\substack{0 \le \alpha_1 \le \cdots \le \alpha_k \\ \alpha_1+\cdots+\alpha_k = i}}\left(\ord_p \left (\frac{k!}{N(\alpha_1,\cdots,\alpha_k)}\right)+\sum_{j=1}^k v_{\alpha_j}(f)\right).\end{equation}
\begin{proof}[Proof of lemma]
\renewcommand{\qedsymbol}{$\Diamond$}
In the expansion of $f(z)^k$, the coefficient of $z^i$ is the sum of terms of the form \[a_{\alpha_1} \cdots a_{\alpha_k},\] where $\alpha_1 + \cdots+\alpha_k = i$. Then for a fixed $k$-tuple $(\alpha_1,\cdots,\alpha_k)$ such that \[0\le\alpha_1 \le \cdots \le \alpha_k\ \ \text{and}\ \ \alpha_1+\cdots+\alpha_k = i,\] there are exactly \[\frac{k!}{N(\alpha_1,\cdots,\alpha_k)}\] terms which appears in the coefficient of $z^i$ in the expansion of $f(z)^k$ which are equal to $a_{\alpha_1} \cdots a_{\alpha_k}$. Therefore the coefficient of $z^i$ is \[\sum_{\substack{0\le \alpha_1\le \cdots\le \alpha_k \\ \alpha_1 + \cdots + \alpha_k=i}} \frac{k!}{N (\alpha_1,\cdots,\alpha_k)} a_{\alpha_1} \cdots a_{\alpha_k}.\] Now the properties of $p$-valuation give \eqref{eq13}.
\end{proof}
\end{lemma}

\section{Irreducibility of $G_1$}
\label{section:g1}

In this section, we prove Theorem \ref{mainthm} for $m=1$. From now on, unless it is specified otherwise, $d$ is an odd prime. First we note that a direct application of Corollary \ref{eisenstein} does not work; for example, from Proposition \ref{pgnga} we have
\begin{equation}
\label{eq2}
    G_1=\frac{1}{a-d} \left(d^d-\frac{a^d-d^d}{a-d}\right),
\end{equation}
and it turns out that the $d$-Newton polygon of $G_1$ is composed of a single line segment between $(0,d-2)$ and $(d-2,0)$, which contains many lattice points.

To apply the theory of Newton polygons, we first do a change of variable \[a \mapsto (b+1)d.\] That is, now we consider the family of maps \[\varphi_b ([x,y]) = [(b+1)dxy^{d-1},x^d+(d-1)y^d]\] where $b \ne -1$. Let \[\varphi_b^{n} ([x,y]) = [R_n(x,y),S_n(x,y)]\] and \[r_n = R_n(1,1),\ \ s_n = S_n(1,1) \in \bZ[b].\] Then as above we can construct the pre-Misiurewicz polynomial $\widetilde{\cG}_m$ and the Misiurewicz polynomial $\cG_m$ in $\bZ[b]$, satisfying the following equations which directly come from Proposition \ref{pgnga}. Note that we used caligraphic fonts to denote that those polynomials are obtained by a change of variable from $\widetilde{G}_m$ and $G_m$.

\begin{proposition}
\label{pgngb}
$\widetilde{\cG}_m$ and $\cG_m$ satisfy the following equations.
\[\widetilde{\cG}_m=(b+1)bd^2 s_{m-1}^{d-1} \cG_m,\] \[\cG_m= \frac{1}{bd} \left[(b+1)^d d^d s_{m-1}^{d(d-1)} - (bd+1) \frac{s_m^d-((b+1)d s_{m-1}^d)^d}{s_m-(b+1)d s_{m-1}^d}\right] \in \bZ[b].\]
\end{proposition}

The irreducibility of $G_m$ over $\bQ$ is the same as that of $\cG_m$, since the change of variable is linear. From now on, we prove the irreducibility of $\cG_m$.

\begin{proposition}
\label{g1}
The first Misiurewicz polynomial $\cG_1$ is irreducible over $\bQ$.
\begin{proof}
From \eqref{eq2}, we have \begin{align*} \cG_1&=\frac{1}{bd} \left(d^d-\frac{(b+1)^d d^d-d^d}{bd}\right) \\&= -\frac{1}{bd} \left(\left(\sum_{i=0}^{d-1} \binom{d}{i+1} b^i d^{d-1} \right)-d^d \right) \\ &= - \sum_{i=0}^{d-2} \binom{d}{i+2} b^{i} d^{d-2}.\end{align*} Since \[\ord_d \left(\binom{d}{i+2}\right) = \begin{cases} 1 & \text{if } i=0,\cdots,d-3, \\ 0 & \text{if }i=d-2, \end{cases}\] we have \[v_{i}(\cG_1) =  \begin{cases} d-1 & \text{if } i=0,\cdots,d-3, \\ d-2 & \text{if }i=d-2. \end{cases}\] It follows that $N_d(\cG_1)$ is composed of a single line segment between $(0,d-1)$ and $(d-2,d-2)$ which has no other lattice point except for the vertices. Therefore, by Eisenstein's criterion $\cG_1$ is irreducible over $\bQ$.
\end{proof}
\end{proposition}

\section{Irreducibility of $G_2$}
\label{section:g2}

In this section, we prove Theorem \ref{mainthm} for $m=2$. From Proposition \ref{pgngb} we have \begin{equation} \label{eq3} bd\cG_2 =(b+1)^d d^d s_1^{d(d-1)}-(bd+1)\frac{s_2^d - ((b+1)ds_1^d)^d}{s_2-(b+1)ds_1^d}.\end{equation} If we let \[\sigma = (b+1)ds_1^d,\ \ \tau=s_2-\sigma = s_2-(b+1)ds_1^d,\] then \[(b+1)^d d^d s_1^{d(d-1)} = (b+1)d \sigma^{d-1}\] and \[\frac{s_2^d - ((b+1)ds_1^d)^d}{s_2-(b+1)ds_1^d} = \frac{(\sigma+\tau)^d-\sigma^d}{\tau} = \sum_{k=0}^{d-1} \binom{d}{k} \sigma^{k} \tau^{d-1-k} .\] Therefore, \begin{equation} \label{eq7} \begin{split} bd \cG_2  &= (b+1)d \sigma^{d-1} - (bd+1) \sum_{k=0}^{d-1} \binom{d}{k} \sigma^k \tau^{d-1-k} \\ & = -bd(d-1) \sigma^{d-1} - (bd+1) \sum_{k=0}^{d-2} \binom{d}{k} \sigma^k \tau^{d-1-k}. \end{split}\end{equation}

\begin{proposition}
\label{g2}
The second Misiurewicz polynomial $\cG_2$ is irreducible over $\bQ$.
\begin{proof}
We first investigate $N_d(\sigma)$ and $N_d(\tau)$. Since \[s_1=d,\ \ s_2 = ((b+1)d)^d+(d-1)d^d=((b+1)^d+d-1)d^d,\] we have \begin{equation} \label{eq8} \sigma=(b+1)d^{d+1},\ \ \tau=d^d \sum_{i=2}^d \binom{d}{i} b^i.\end{equation} Therefore $N_d(\sigma)$ is defined by two vertices $(0,d+1)$ and $(1,d+1)$, while $N_d(\tau)$ is defined by three vertices $(0,\infty)$, $(2,d+1)$, and $(d,d)$.

Now we investigate $N_d(bd \cG_2)$ with \eqref{eq7} and the $d$-Newton polygons of $\sigma$ and $\tau$ from above. Explicitly, we claim the following:
\begin{enumerate}[(i)]
    \item $bd\cG_2$ is divisible by $b$, and $v_1(bd\cG_2)=d^2$.
    \item The degree of $bd\cG_2$ is $d^2-d+1$, and \[v_{d^2-d}(bd\cG_2) = d^2-d,\ \ v_{d^2-d+1}(bd\cG_2)=d^2-d+1.\]
    \item Let $\ell$ be the line in the $xy$-plane passing through the two points $(1,d^2)$ and $(d^2-d,d^2-d)$. Then $(i,v_i(bd\cG_2))$ is on or above the line $\ell$ for all $i=1,2,\cdots,d^2-d$.
\end{enumerate}

For (i), we observe that $\tau$ is divisible by $b^2$. Therefore, \begin{align*} bd\cG_2 & \equiv -bd(d-1)\sigma^{d-1}  \\& = -bd(d-1)(b+1)^{d-1}d^{d^2-1} \equiv -bd^{d^2} (d-1) \pmod{b^2}.\end{align*} This proves (i).

For (ii), we observe that $\deg (\sigma)=1$ while $\deg(\tau)=d$. Therefore, \[\deg(bd(d-1)\sigma^{d-1} )=d\] and \[\deg \left((bd+1)\binom{d}{k} \sigma^k \tau^{d-1-k}\right) = 1+k+d(d-1-k)=d^2-d+1 - (d-1)k\] for each $k=0,\cdots,d-2$. With \eqref{eq7}, this implies that $\deg (bd\cG_2)=d^2-d+1$, and moreover terms with degree $d^2-d+1$ and $d^2-d$ come from the $k=0$ case only. It gives \begin{align*} -(bd+1)\tau^{d-1} & = -(bd+1) d^{d^2-d} (b^d+db^{d-1}+\cdots)^{d-1} \\ &= -(bd+1) d^{d^2-d} (b^{d^2-d}+(d^2-d) b^{d^2-d-1}+\cdots) \\ &= -d^{d^2-d} (db^{d^2-d+1} + (d^3-d^2+1)b^{d^2-d}+\cdots), \end{align*}  and (ii) directly follows.

To prove (iii), we investigate the $d$-Newton polygon of each term in \eqref{eq7}. First, the $d$-Newton polygon of \[bd(d-1) \sigma^{d-1}\] is defined by three vertices $(0,\infty)$, $(1,d^2)$, and $(d-1,d^2)$, and all vertices are above $\ell$. On the other hand, if $k\ne 0$, then the $d$-Newton polygon of \[(bd+1) \binom{d}{k} \sigma^k \tau^{d-1-k}\] is defined by five vertices \[(0,\infty),\ \ (2(d-1-k),d^2),\ \ (d(d-1-k),d^2-(d-1-k)),\] \[(d(d-1-k)+k,d^2-(d-1-k)),\ \ (d(d-1-k)+k+1,d^2-(d-1-k)+1).\] Note that $\binom{d}{k}$ has $d$-adic valuation $1$ for $k=1,\cdots,d-2$, which increases the total $d$-adic valuation by $1$. All of these vertices are on or above $\ell$. Finally, even if $k=0$, the $d$-Newton polygon of \[(bd+1) \tau^{d-1}\] is defined by four vertices \[(0,\infty),\ \ (2(d-1),d^2-1),\ \ (d^2-d,d^2-d),\ \ (d^2-d+1,d^2-d+1),\] and all of these vertices are on or above $\ell$. Therefore, since the $d$-adic valuation satisfies the non-Archimedean triangle inequality, $N_d(bd\cG_2)$ is on or above $\ell$ as well. This proves (iii).

Now (i), (ii), and (iii) says that $N_d(bd\cG_2)$ is defined by four vertices \[(0,\infty),\ \ (1,d^2),\ \ (d^2-d,d^2-d),\ \ (d^2-d+1,d^2-d+1),\] or equivalently $N_d(\cG_2)$ is defined by three vertices \[(0,d^2-1),\ \ (d^2-d-1,d^2-d-1),\ \ (d^2-d,d^2-d).\]

Since the line segment between $(0,d^2-1)$ and $(d^2-d-1,d^2-d-1)$ contains no lattice point except for the endpoints, Corollary \ref{newtonfactor} says that $\cG_2$ has at most two factors over $\bQ$, so also over $\bZ$. Moreover, in the proof of Corollary \ref{newtonfactor}, if $\cG_2$ has two factors then one should be associated to the line segment between $(d^2-d-1,d^2-d-1)$ and $(d^2-d,d^2-d)$. This factor should be linear, whose $d$-Newton polygon is composed of a single line segment of slope $1$. Since we calculated above that \[bd\cG_2=-d^{d^2}(d-1)b-\cdots-d^{d^2-d+1}b^{d^2-d+1}\] so \[\cG_2 = -d^{d^2-1}(d-1) - \cdots - d^{d^2-d} b^{d^2-d}=-d^{d^2-d-1} (-d^d (d-1) - \cdots - db^{d^2-d}).\] This means that the only possible factors are \[(\text{a divisor of }(d-1))\pm db.\]

We claim any such linear polynomial cannot be a factor of $\cG_2$, so $\cG_2$ is indeed irreducible over $\bQ$. It is equivalent to show that any $b=e/d$, where $e$ is a (positive or negative) divisor of $d-1$, cannot be a root of $\cG_2$. For the sake of contradiction, suppose that there is a root $b=e/d$ such that $\cG_2(b)=0$. For such $b$, \eqref{eq7} says that \begin{equation*} bd(d-1)\sigma^{d-1} = -(bd+1) \sum_{k=0}^{d-2} \binom{d}{k} \sigma^k \tau^{d-1-k}\end{equation*} so \begin{equation} \label{eq9} e(d-1) \sigma^{d-1} = -(e+1) \sum_{k=0}^{d-2} \binom{d}{k} \sigma^k \tau^{d-1-k}. \end{equation} From \eqref{eq8}, we have \[\ord_d(\sigma) = \ord_d ((d+e)d^d) =d,\] so the left hand side of \eqref{eq9} has a $d$-adic valuation of $d(d-1)$.

On the other hand, also from \eqref{eq8} we have \begin{align*} \ord_d (\tau) & = \ord_d \left( d^d \left(b^d+db^{d-1}+\cdots+ \binom{d}{2} b^2 \right)\right) \\&= \ord_d \left(e^d+d^2 e^{d-1}+\cdots + d^{d-2} \binom{d}{2} e^2 \right) = 0.\end{align*} Therefore, if $0<k<d$ then \[\ord_d \left( \binom{d}{k} \sigma^k \tau^{d-1-k} \right) = 1+dk,\] while \[\ord_d \left(\binom{d}{0} \sigma^0 \tau^{d-1} \right) = \ord_d (\tau^{d-1})=0.\] This implies \[\ord_d \left(\sum_{k=0}^{d-2} \binom{d}{k} \sigma^k \tau^{d-1-k}\right) = 0.\] Since $e$ is a divisor of $d-1$, $\ord_d(e+1)\le 1$ unless $e=-1$. Hence the $d$-adic valuation of the right hand side of \eqref{eq9} must be either $\infty$ or at most $1$. Therefore \eqref{eq9} cannot be true, contradiction. This ends the proof.
\end{proof}
\end{proposition}

\section{Irreducibility of $G_3$}
\label{section:g3}

Now we prove that $\cG_3$ is irreducible over $\bQ$, which directly implies the irreducibility of $G_3$ as well. That will finish the proof of Theorem \ref{mainthm}. As in the previous section, Proposition \ref{pgngb} gives \begin{equation} \label{eq10}\cG_3=\frac{1}{bd} \left[(b+1)^d d^d s_2^{d(d-1)}-(bd+1)\frac{s_3^d - ((b+1)ds_2^d)^d}{s_3-(b+1)ds_2^d}\right].\end{equation} Similarly, we let \[\sigma = (b+1)d s_2^d,\ \ \tau = s_3-\sigma= s_3 - (b+1)d s_2^d.\] Then as in \eqref{eq7}, we have \begin{equation} \label{eq16} bd\cG_3 = -bd(d-1)\sigma^{d-1} - (bd+1) \sum_{k=0}^{d-2} \binom{d}{k} \sigma^k \tau^{d-1-k}.\end{equation}

Before we prove the irreducibility of $\cG_3$, we investigate $N_d(\sigma)$ and $N_d(\tau)$. First, \[\sigma = (b+1)d s_2^d = (b+1)dF,\] where \[F=s_2^d=((b+1)^d+d-1)^d d^{d^2}.\] It follows that $N_d(\sigma)$ is defined by three vertices $(0,d^2+d+1)$, $(d^2,d^2+1)$, and $(d^2+1,d^2+1)$.

$\tau$ is more complicated. We can first observe that \[r_2 =(b+1)^2 d^{d+1},\] so \begin{equation} \label{eq12} \tau = s_3-\sigma = r_2^d+(d-1)s_2^d -(b+1)d s_2^d =(b+1)^{2d} d^{d^2+d}-(bd+1)F.\end{equation}

\begin{proposition}
\label{taunp}
The $d$-Newton polygon of $\tau$ is defined by five vertices \[(0,\infty),\ \ (3,d^2+d+1),\ \ (d+1,d^2+d),\ \ (d^2,d^2),\ \ (d^2+1,d^2+1).\]
\begin{proof}
We first directly calculate the terms in $\tau$ with degree at most $3$ to show that \[\tau = -\frac{(d-1)^2}{2}d^{d^2+d+1}b^3+(\text{higher terms}),\] which gives the first two vertices. Also, since $\tau$ is equal to $-(bd+1)F$ for the terms with degree at least $2d+1$, the last two vertices follow from $N_d((bd+1)F)$.

For the remaining vertex $(d+1,d^2+d)$, we use Lemma \ref{orderlemma} from Section \ref{section:newton}. In particular, it turns out that if $k=d$ is prime in the lemma and $d \nmid i$ then \[v_i(f^d)\ge 1+\min_{\substack{0\le \alpha_1 \le \cdots \le \alpha_d \\ \alpha_1 + \cdots + \alpha_d = i}} \left(\sum v_{\alpha_j} (f)\right).\] If $d\mid i$, then letting $i=de$, due to the possibility of $\alpha_1 = \cdots = \alpha_k$ we have \[v_{de}(f^d)\ge \min\left[ 1+\min_{\substack{0\le \alpha_1 \le \cdots \le \alpha_d \\ \alpha_1 + \cdots + \alpha_d = de}} \left(\sum v_{\alpha_j} (f)\right),d v_{e}(f) \right].\]

Applying this to $f=(b+1)^d+d-1$, where $v_i(f)=0$ for $i=d$ and $v_i(f)=1$ for all $i<d$, we can conclude that \[v_i (f^d) \begin{cases} \ge d+1-\left\lfloor \frac{i}{d}\right\rfloor & \text{if } d \nmid i, \\ \ge \min(d+1-e,d) & \text{if } i=de,\ e<d, \\ =0 & \text{if } i=d^2. \end{cases}\] Note that \[d \ge d+1-e = d+1-\left \lfloor \frac{i}{d}\right \rfloor\] for all $i=de$ where $0<e<d$. Moreover, since the constant term of $f$ is $d$, \[v_0 (f^d) = \ord_d(d^d)= d.\] Therefore, we can simplify the above inequality to \begin{equation} \label{eq11} v_i(f^d) \begin{cases} =0 & \text{if }i=d^2, \\ =d & \text{if } i=0, \\ \ge d+1-\left \lfloor \frac{i}{d} \right \rfloor & \text{otherwise.} \end{cases}\end{equation}

Now we claim the following:

\begin{enumerate}[(i)]
    \item $v_i(\tau) \ge d^2+d+1$ for $i=4,\cdots,d$
    \item $v_{d+1}(\tau) = d^2+d$.
    \item Let $\ell$ be the line in the $xy$-plane passing through the two points $(d+1,d^2+d)$ and $(d^2,d^2)$. Then $(i,v_i(\tau))$ is on or above the line $\ell$ for all $i=d+2,\cdots,d^2-1$.
\end{enumerate}

For (i), when $i=4,\cdots,d-1$ we have \[v_i(f^d)\ge d+1,\] so \begin{align*} v_i ((bd+1)F) & \ge \min(v_{i-1}(F)+1,v_i(F)) \\ &= \min (v_{i-1} (f^d)+d^2+1,v_i(f^d)+d^2) ) \ge d^2+d+1.\end{align*} On the other hand, \[v_i((b+1)^{2d} d^{d^2+d}) = \ord_d \left(\binom{2d}{i}\right) + (d^2+d) = d^2+d+1\] so from \eqref{eq12} we have \[v_i(\tau)\ge d^2+d+1.\]

When $i=d$, the coefficient of $b^d$ in \eqref{eq12} is \begin{equation} \label{eq14} d^{d^2+d}\binom{2d}{d}  - (d c_{d-1} (F) + c_d(F))=d^{d^2+d}\binom{2d}{d}  - d^{d^2+1} c_{d-1} (f^d) -d^{d^2} c_d(f^d).\end{equation} We have shown above that \[\ord_d(c_{d-1}(f^d)) = v_{d-1} (f^d) \ge d+1,\] so the second term has $d$-adic valuation at least $d^2+d+2$. For $c_d(f^d)$, considering \eqref{eq13} in Lemma \ref{orderlemma}, the least possible $d$-valuation $d$ appears only when $\alpha_1 = \cdots = \alpha_d = 1$ or $\alpha_1 = \cdots =\alpha_{d-1} =0$ and $\alpha_d=d$. In other words, \begin{align*} c_d (f^d) &= \underbrace{d^d}_{\alpha_1 = \cdots = \alpha_d = 1} + \underbrace{d \cdot d^{d-1}}_{\alpha_1 = \cdots = \alpha_{d-1}=0,\ \ \alpha_d = d} + (\text{divisible by }d^{d+1})\\ &= 2d^d + (\text{divisible by }d^{d+1}).\end{align*} Therefore, from \eqref{eq14}, \[c_d(\tau) = d^{d^2+d} \left(\binom{2d}{d} - 2\right) + (\text{divisible by }d^{d^2+d+1}).\] In fact, $c_d(\tau)$ must be divisible by $d^{d^2+d+1}$ since $\binom{2d}{d} \equiv 2 \pmod d$. This proves that $v_d (\tau) \ge d^2+d+1$ as well.

For (ii), the coefficient of $b^{d+1}$ in \eqref{eq12} is \begin{equation}\begin{aligned} \label{eq15} d^{d^2+d} \binom{2d}{d+1} & - (dc_d(F)+c_{d+1}(F)) \\& = d^{d^2+d} \binom{2d}{d+1} - d^{d^2+1} c_d(f^d) - d^{d^2} c_{d+1} (f^d),\end{aligned}\end{equation} and it turns out that the first two terms have $d$-adic valuation $d^2+d+1$. On the other hand, in $c_{d+1}(f^d)$ the least possible $d$-valuation $d$ appears only when $\alpha_1 = \cdots = \alpha_{d-2}=0$, $\alpha_{d-1}=1$, and $\alpha_d = d$ in \eqref{eq13}, so \[c_{d+1} (f^d)= d(d-1) d^{d-1} + (\text{divisible by }d^{d+1})=(d-1)d^d + (\text{divisible by }d^{d+1})\] has $d$-adic valuation exactly $d$. This implies that the right hand side of \eqref{eq15} has $d$-adic valuation exactly $d^2+d$, which implies (ii).

For (iii), we observe that \[v_i((b+1)^{2d}d^{d^2+d}) \ge d^2+d\] so it suffices to prove the same statement for $(bd+1)F$ instead of $\tau$. However, using \eqref{eq11} we have \begin{align*} v_i((bd+1)F) &\ge \min(v_{i-1}(F)+1,v_i(F)) \\ & = \min(v_{i-1}(f^d)+d^2+1,v_i(f^d)+d^2) \ge d^2+d+1 - \left \lfloor \frac{i}{d}\right \rfloor \end{align*} for all $i=d+2,\cdots,d^2-1$, and for such $i$, $(i,d^2+d+1-\left \lfloor \frac{i}{d}\right \rfloor)$ is on or above $\ell$. This proves (iii).

Now (i), (ii), and (iii) says $(d+1,d^2+d)$ is the only other vertex in $N_d(\tau)$.
\end{proof}
\end{proposition}

Now we are ready to prove our main result.

\begin{proof}[Proof of Theorem \ref{mainthm}]
As in Proposition \ref{g2}, we investigate $N_d(bd\cG_3)$ with \eqref{eq16} and the $d$-Newton polygons of $\sigma$ and $\tau$. Explicitly, we claim the followings:
\begin{enumerate}[(i)]
    \item $bd\cG_3$ is divisible by $b$, and $v_1(bd\cG_3)=d^3$.
    \item The degree of $bd\cG_3$ is $d^3-d^2$, and \[v_{d^3-d^2}(bd\cG_3) = d^3-d^2.\]
    \item Let $\ell$ be the line in the $xy$-plane passing through the two points $(1,d^3)$ and $(d^3-d^2,d^3-d^2)$. Then $(i,v_i(bd\cG_3))$ is on or above the line $\ell$ for all $i=1,2,\cdots,d^3-d^2$.
\end{enumerate}
For (i), we observe that $\tau$ is divisible by $b^2$. Therefore, \begin{align*} bd\cG_3 & \equiv -bd(d-1)\sigma^{d-1} \\ & = -bd(d-1)((b+1)dF)^{d-1} \\ & = -bd(d-1)((b+1)((b+1)^d+d-1)^d d^{d^2+1})^{d-1} \\& \equiv -bd(d-1) (d^{d^2+d+1})^{d-1} = -bd^{d^3}(d-1) \pmod{b^2}.\end{align*} This proves (i).

For (ii), we observe that $\deg (\sigma) = \deg (\tau) = d^2+1$, so \[\deg (\sigma^k \tau^{d-1-k}) = (d^2+1)(d-1)=d^3-d^2+d-1\] for all $k$. However, in the expression \eqref{eq12} of $\tau$, $(b+1)^{2d} d^{d^2+d}$ has degree $2d$, so any term in $\sigma^k \tau^{d-1-k}$ which comes from multiplying $(b+1)^{2d} d^{d^2+d}$ has degree at most \[(d^3-d^2+d-1)-(d^2-2d+1) = d^3-2d^2+3d-2 < d^3-d^2.\] In other words, $(b+1)^{2d} d^{d^2+d}$ cannot affect the terms in \eqref{eq16} with degree at least $d^3-d^2$, or equivalently those terms are the same when we replace $\tau$ by $-(bd+1)F$ in \eqref{eq16}. However, then \eqref{eq16} becomes \begin{align*}  &-bd(d-1)\sigma^{d-1} -  (bd+1) \sum_{k=0}^{d-2} \binom{d}{k} \sigma^k (-(bd+1)F)^{d-1-k}\\ =& -bd(d-1) ((b+1)dF)^{d-1} \\&- (bd+1) \sum_{k=0}^{d-2} \binom{d}{k} ((b+1)dF)^k (-(bd+1)F)^{d-1-k}  \\ =& F^{d-1} \Biggl[-bd(d-1)((b+1)d)^{d-1} \Biggr. \\& \Biggl.- (bd+1) \sum_{k=0}^{d-2} \binom{d}{k} ((b+1)d)^k (-(bd+1))^{d-1-k}\Biggr] \\ =& F^{d-1} \left[ -bd(d-1)((b+1)d)^{d-1} + \sum_{k=0}^{d-2} \binom{d}{k} ((b+1)d)^k (-(bd+1))^{d-k} \right] \\ =& F^{d-1} \sum_{k=0}^d \binom{d}{k} ((b+1)d)^k (-(bd+1))^{d-k} \\ =& F^{d-1} ((b+1)d-(bd+1))^d \\ =& F^{d-1} (d-1)^d=((b+1)^d+d-1)^{d^2-d} d^{d^3-d^2} (d-1)^d.\end{align*} The leading term is $d^{d^3-d^2} (d-1)^d b^{d^3-d^2}$, so $bd\cG_3$ also has degree $d^3-d^2$ and $v_{d^3-d^2} (bd\cG_3) = d^3-d^2$. This proves (ii).

Finally, for (iii), we investigate the $d$-Newton polygon of each term in \eqref{eq16}. First, the $d$-Newton polygon of $bd(d-1)\sigma^{d-1}$ is defined by four vertices \[(0,\infty),\ \ (1,d^3),\ \ (d^3-d^2+1,d^3-d^2+d),\ \ (d^3-d^2+d,d^3-d^2+d),\] and all of these vertices are on or above $\ell$. On the other hand, if $k\ne 0$, then the $d$-Newton polygon of \[(bd+1)\binom{d}{k} \sigma^k \tau^{d-1-k}\] is defined by seven vertices \[(0,\infty),\ \ (3(d-1-k),d^3),\ \ ((d+1)(d-1-k),d^3-(d-1-k)),\] \[(d^2(d-1-k),d^3-(d+1)(d-1-k)),\] \[(d^2(d-1),d^3-(d+1)(d-1-k)-dk),\] \[ (d^2(d-1)+k,d^3-(d+1)(d-1-k)-dk),\] \[(d^2(d-1)+d,d^3-(d+1)(d-1-k)-dk+d-k),\] (note that $\binom{d}{k}$ gives the additional $d$-valuation of $1$) and all of these vertices are on or above $\ell$. Even if $k =0$, the $d$-Newton polygon of $(bd+1)\tau^{d-1}$ is defined by four vertices \[(3(d-1),d^3-1),\ \ ((d+1)(d-1),d^3-d),\] \[(d^2(d-1),d^3-d^2),\ \ (d^2(d-1)+d,d^3-d^2+d),\] and all of these vertices are on or above $\ell$. Then the non-Archimedean triangle inequality of the $d$-adic valuation proves (iii).

Now (i), (ii), and (iii) say that $N_d(bd\cG_3)$ is defined by three vertices $(0,\infty)$, $(1,d^3)$, and $(d^3-d^2,d^3-d^2)$, or equivalently $N_d(\cG_3)$ is defined by two vertices $(0,d^3-1)$ and $(d^3-d^2-1,d^3-d^2-1)$. Since the line segment between those two points contains no lattice point except for the endpoints, Corollary \ref{eisenstein} says that $\cG_3$ is irreducible over $\bQ$.
\end{proof}

\section*{Acknowledgement}

The author thanks Joseph H. Silverman for his advice on setting the topic and helpful comments. The author also appreciates the referee for their careful review and detailed comments and suggestions. They helped a lot to improve the paper.


\end{document}